\documentclass[10pt,leqno]{amsart}
\usepackage{amsmath}%
\usepackage{amsfonts}%
\usepackage{amssymb}%
\usepackage{graphicx}
\newtheorem{theorem}{Theorem}
\theoremstyle{plain}

\newtheorem{lemma}{Lemma}

\newtheorem{preremark}{Remark}

\numberwithin{equation}{section}

\begin{document}

\title[Energy Conservation by Weak Solutions]{On the energy conservation by weak solutions of the relativistic Vlasov-Maxwell system}
\author{Reinel Sospedra-Alfonso}
\address
{Department of Mathematics and Statistics\newline%
\indent University of Victoria, PO BOX 3045 STN CSC, Victoria BC V8W 3P4}%
\email[R. Sospedra-Alfonso]{sospedra@math.uvic.ca}%
\keywords{Vlasov-Maxwell, weak solutions, conservation of the total energy}%

\begin{abstract}
We show that weak solutions of the relativistic Vlasov-Maxwell system preserve the total energy provided that the electromagnetic field is locally of bounded variation and, for any $\lambda>0$, the one-particle distribution function has a square integrable $\lambda$-moment in the momentum variable.
\end{abstract}
\maketitle


\section{Introduction}\label{intro}
Consider an ensemble of relativistic charged particles that interact through their self-induced electromagnetic field. If collisions among the particles are so improbable that they can be neglected, then the ensemble can be modeled by the so-called relativistic Vlasov-Maxwell (RVM) system. At any given time $t\in]0,\infty[$, the RVM system is characterized by the one-particle distribution function $f=f(t,x,p)$ with position $x\in\mathbb{R}^3$ and momentum $p\in\mathbb{R}^3$. The self-induced electric and magnetic fields are denoted by $E=E(t,x)$ and $B=B(t,x)$, respectively. Setting all physical constants to one, the model equations for a single particle species read  
\begin{equation}
\frac{\partial f}{\partial t}+v\cdot\nabla_{x}f+\left(E+v\times B\right)\cdot\nabla_{p}f=0
\label{Vlasov Equation}
\end{equation}
\begin{equation}
\frac{\partial E}{\partial t}-\nabla\times B=-4\pi j
\label{Maxwell Evolution 1}
\end{equation}
\begin{equation}
\frac{\partial B}{\partial t}+\nabla\times E=0
\label{Maxwell Evolution 2}
\end{equation}
\begin{equation}
 \begin{tabular}{cc}
    $\nabla\cdot E=4\pi\rho$, & $\nabla\cdot B=0$,
 \end{tabular}
\label{Maxwell Constraints}
\end{equation}
where $v:=p\left(1+|p|^2\right)^{-1/2}$ denotes the relativistic velocity. The coupling of the Vlasov (\ref{Vlasov Equation}) and Maxwell equations (\ref{Maxwell Evolution 1})-(\ref{Maxwell Constraints}) is through the charge and current densities, which we denoted by $\rho=\rho(t,x)$ and $j=j(t,x)$ respectively. They are defined by 
\begin{equation}
 \begin{tabular}{cc}
  $\rho:=\int_{\mathbb{R}^3}fdp$, & $j:=\int_{\mathbb{R}^3}vfdp$.
 \end{tabular}
\label{Density and Current}
\end{equation}
We define the Cauchy problem for the RVM system by (\ref{Vlasov Equation})-(\ref{Density and Current}) with initial data
\begin{equation}
 \begin{tabular}{ccc}
   $f_{|t=0}=f_0$, & $E_{|t=0}=E_0$, & $B_{|t=0}=B_0$,
 \end{tabular}
\label{Initial Data} 
\end{equation}
satisfying (\ref{Maxwell Constraints}) in the sense of distribution. It is not difficult to check that if (\ref{Maxwell Constraints}) holds at $t=0$, then it will do so for all time in which the solution exist. Thus, the equations (\ref{Maxwell Constraints}) can be understood as a mere constraint on the initial data.  

Now, define  
\begin{displaymath}
 L^1_{kin}(\mathbb{R}^6):=\left\{g\in L^1(\mathbb{R}^6):g\geq0, \int\int\sqrt{1+\left|p\right|^2}g(x,p)dxdp<\infty\right\}.
\end{displaymath}
For $T>0$, we say that $(f,\;E,\;B)$ is a weak solution of the RVM system if
\begin{equation}
\label{Space Weak Solution}
\begin{tabular}{cc}
$f\in L^{\infty}([0,T[;L^1_{kin}\cap L^{\infty}(\mathbb{R}^6))$, & $E, B\in \left[L^{\infty}([0,T[;L^2({\mathbb{R}^3}))\right]^3$
\end{tabular}
\end{equation}
and the equations (\ref{Vlasov Equation})-(\ref{Maxwell Constraints}) are satisfied in the sense of distributions. In particular, we say that the Vlasov equation (\ref{Vlasov Equation}) is satisfied in the sense of distributions if for all $\varphi\in C^{\infty}([0,T]\times\mathbb{R}^6)$ with compact support in $[0,T[\times\mathbb{R}^6$
\begin{eqnarray}
\label{Vlasov Equation Weak Form}
\int^T_0\int_{\mathbb{R}^3\times\mathbb{R}^3}f(t,x,p)\left[\partial_t\varphi+v\cdot\nabla_x\varphi+K\cdot\nabla_p\varphi\right](t,x,p)dtdxdp\nonumber\\
=-\int_{\mathbb{R}^3\times\mathbb{R}^3}f_0(x,p)\varphi(0,x,p))dxdp.
\end{eqnarray}
We define analogous relations for the Maxwell equations (\ref{Maxwell Evolution 1})-(\ref{Maxwell Constraints}) as well. The vector field $K:=E+v\times B$ in (\ref{Vlasov Equation Weak Form}) denotes the Lorentz force acting on a reference particle of velocity $v$. Notice that it satisfies $\nabla_p\cdot K\equiv0$.

The global existence result for weak solutions in both relativistic and non-relativistic settings is due to DiPerna and Lions and can be found in \cite{DiPernaLions}. In \cite{Rein} this result is revisited. The uniqueness problem, on the other hand, remains unsolved. It is also unknown whether weak solutions preserve the total energy at least almost everywhere in time, cf. \cite[Remark 4, p.740]{DiPernaLions}. It is known, however, that the energy is bounded at almost all $t$ by its value at $t=0$, namely
\begin{eqnarray}
\label{Bound Total Energy}
\mathcal{E}(t)&:=&\int_{\mathbb{R}^3\times\mathbb{R}^3}\sqrt{1+\left|p\right|^2}f(t,x,p)dxdp\nonumber\\
& & +\frac{1}{8\pi}\int_{\mathbb{R}^3}\left|E(t,x)\right|^2+\left|B(t,x)\right|^2dx\hspace{.2cm}\leq\hspace{.2cm}\mathcal{E}(0).
\end{eqnarray}

In the present note we show that if the electric and magnetic fields $E$ and $B$ are locally of bounded variation and, for any $\lambda>0$, the function  
\begin{equation}
\label{Rho Lambda}
\rho_{\lambda}(t,x):=\int_{\mathbb{R}^3}\left|p\right|^{\lambda}f(t,x,p)dp
\end{equation}
is square integrable, then the relation (\ref{Bound Total Energy}) is in fact an equality for almost all $0\leq t<T$. Precisely, we prove the following result:
\begin{theorem} 
\label{The Theorem}
Let $\lambda>0$. Let $f_0\in L^1_{kin}\cap L^{\infty}(\mathbb{R}^6)$ , $E_0,B_0\in \left[L^2({\mathbb{R}^3})\right]^3$ and denote by $(f,E,B)$ a weak solution of the RVM system satisfying $\left.(f,E,B)\right|_{t=0}=(f_0,E_0,B_0)$. If $E,B\in [L^1_{loc}(]0,T[;BV_{loc}(\mathbb{R}^3))]^3$ and $\rho_{\lambda}$ as defined in (\ref{Rho Lambda}) is in $L^{\infty}_{loc}(]0,T[;L^2(\mathbb{R}^3))$, then the total energy defined by (\ref{Bound Total Energy}) satisfies $\mathcal{E}(t)=\mathcal{E}(0)$ for almost all $0\leq t<T$.
\end{theorem}

The tools we use are basically those introduced by DiPerna and Lions in \cite{Duality} to study renormalized solutions of transport equations. We shall also refer to \cite{Bouchut}, where applications to the Vlasov equation are given. We remark that the same result holds for the electromagnetic field in $[L^1_{loc}(]0,T[;W^{1,1}_{loc}(\mathbb{R}^3))]^3$ since we have the (strict) inclusion $W^{1,1}(\Omega)\subset BV(\Omega)$ for any open set $\Omega\subseteq\mathbb{R}^3$. For a detailed account on functions of bounded variation cf. \cite{Ambrosio}. We would like to include here the reference \cite{Loeper}, where the uniqueness of weak solutions for the Vlasov-Poisson system has been obtained under the sole assumption that the \textit{spatial} density is bounded. Similar results would be desirable for the more demanding Vlasov-Maxwell system.    

Formally, the law of the conservation of the total energy is derived as follows. Multiply the Maxwell equations (\ref{Maxwell Evolution 1}) and (\ref{Maxwell Evolution 2}) by $E$ and $B$ respectively and integrate on $\mathbb{R}^3$ to find that
\begin{equation}
\label{Formal Energy Maxwell}
\frac{1}{8\pi}\frac{d}{dt}\int_{\mathbb{R}^3}\left|E\right|^2+\left|B\right|^2dx=-\int_{\mathbb{R}^3}j\cdot Edx.
\end{equation}
Multiply the Vlasov equation by $\sqrt{1+\left|p\right|^2}$ and integrate on $\mathbb{R}^3\times\mathbb{R}^3$ to get
\begin{equation}
\label{Formal Energy Vlasov}
\frac{d}{dt}\int_{\mathbb{R}^3\times\mathbb{R}^3}\sqrt{1+\left|p\right|^2}fdxdp=\int_{\mathbb{R}^3}j\cdot Edx.
\end{equation}
Then the sum of (\ref{Formal Energy Maxwell}) and (\ref{Formal Energy Vlasov}) provide the desired result.

As for weak solutions, we shall follow the same scheme. We find relations analogous to (\ref{Formal Energy Maxwell}) and (\ref{Formal Energy Vlasov}) in sections \ref{Section Maxwell} and \ref{Section Vlasov} respectively. The difficulty is to overcome the lack of regularity and the need of justifying the operations taken for granted when the solutions are smooth.

\section{Energy balance for the Maxwell equation}
\label{Section Maxwell}

Here we show that if the current $j$ is square integrable for almost all time, then the weak solution of the RVM system satisfies the energy balance associated to the Maxwell equations, i.e., the relation (\ref{Formal Energy Maxwell}). This result is reminiscent of the duality theorem for transport equations given by DiPerna and Lions in \cite{Duality}.

\begin{lemma} 
\label{Lemma 1}
Let $(f,E,B)$ be a weak solution of the RVM system with initial data $(f_0,E_0,B_0)$. If $j$ as defined in (\ref{Density and Current}) is in $[L^{\infty}(]0,\;T[;L^2(\mathbb{R}^3))]^3$, then 
\begin{equation}
\frac{1}{8\pi}(\left\|E(t)\right\|^2_{L^2_x} + \left\|B(t)\right\|^2_{L^2_x}) + \int^t_0\int_{\mathbb{R}^3}j\cdot Edsdx=\frac{1}{8\pi}(\left\|E_0\right\|^2_{L^2_x} + \left\|B_0\right\|^2_{L^2_x})
\label{Poynting}
\end{equation}
for almost all $t\in[0,T[$.
\end{lemma}


\begin{proof}
Let $\epsilon>0$ and let $\kappa\in C^{\infty}_0(\mathbb{R}^3)$, $\kappa$ even, be a standard mollifier. Define the regularization kernel $\kappa_{\epsilon}:=\frac{1}{\epsilon^3}\kappa(\frac{x}{\epsilon})$. Since mollification and distributional differentiation commute, i.e., $\left(\partial_x u\right)\ast\kappa_{\epsilon}=\partial_x\left(u\ast\kappa_{\epsilon}\right)$, we can convolute (\ref{Maxwell Evolution 1}) and (\ref{Maxwell Evolution 2}) with $\kappa_{\epsilon}$ to obtain 

\begin{eqnarray}
\label{Modified Maxwell Equation 1}
\frac{\partial E_{\epsilon}}{\partial t}-\nabla\times B_{\epsilon} & = & -4\pi j_{\epsilon}\\
\label{Modified Maxwell Equation 2}
\frac{\partial B_{\epsilon}}{\partial t}+\nabla\times E_{\epsilon} & = & 0,
\end{eqnarray}
where $j_{\epsilon}:=j\ast\kappa_{\epsilon}$ and $(E_{\epsilon},B_{\epsilon}):=(E,B)\ast\kappa_{\epsilon}$. 

Consider the family of smooth cut-off functions $\phi_R:=\phi(\frac{\cdot}{R})$, $R\geq1$ where $\phi\in C^{\infty}_0(\mathbb{R}^3)$, $\phi\geq0$ and $\phi\equiv1$ on the ball $B_1$ $\subset$ supp$\phi$ $\subset B_2$. The smoothness of the fields $B_{\epsilon}$ and $E_{\epsilon}$ with respect to $x$ imply via (\ref{Modified Maxwell Equation 1}) and (\ref{Modified Maxwell Equation 2}) that $\partial_t E_{\epsilon},\partial_t B_{\epsilon}\in L^1_{loc}\left[(]0,T[\times\mathbb{R}^3)\right]^3$. Thus, $E_{\epsilon},B_{\epsilon}\in \left[W^{1,1}_{loc}(]0,T[\times\mathbb{R}^3)\right]^3$ and we can apply the chain rule in Sobolev spaces, i.e., for almost all $t\in]0,T[$ 

\begin{displaymath}
\frac{1}{2}\frac{\partial}{\partial t}\left(\left|E_{\epsilon}(t)\right|^2+\left|B_{\epsilon}(t)\right|^2\right)= E_{\epsilon}\cdot\frac{\partial E_{\epsilon}}{\partial t}+B_{\epsilon}\cdot\frac{\partial B_{\epsilon}}{\partial t}.
\end{displaymath}  

Therefore, we can multiply (\ref{Modified Maxwell Equation 1}) and (\ref{Modified Maxwell Equation 2}) by $E_{\epsilon}\phi_R$ and $B_{\epsilon}\phi_R$ respectively, sum the resultant equations and integrate by parts to find that

\begin{eqnarray}
\frac{1}{8\pi}\int_{\mathbb{R}^3}\left(\left|E_{\epsilon}(t)\right|^2+\left|B_{\epsilon}(t)\right|^2\right)\phi_R - \frac{1}{8\pi}\int_{\mathbb{R}^3}\left(\left|E_{\epsilon}(0)\right|^2+\left|B_{\epsilon}(0)\right|^2\right)\phi_R \nonumber\\
=\frac{1}{4\pi}\int^t_0\int_{\mathbb{R}^3}\left(E_{\epsilon}\times B_{\epsilon}\right)\cdot\nabla\phi_R-\int^t_0\int_{\mathbb{R}^3}j_{\epsilon}\cdot E_{\epsilon}\phi_R.
\end{eqnarray}  

Let $\epsilon\rightarrow0$. The terms on the left side converge as a consequence of the theorem of smooth approximations \cite[Theorem 3, p.196]{McOwen}. Also, the same theorem and the assumption made on the current $j$ easily implies that $\int j_{\epsilon}\cdot E_{\epsilon}\rightarrow \int j\cdot E$ for almost all $s\in[0,T[$. Thus, we may invoke the Lebesgue dominated convergence theorem and the convergence of the second term in the right side follows as well. Clearly, the same reasoning applies to the remaining term. Then, for almost all $t\in[0,T[$

\begin{eqnarray}
\label{Sequence Poynting}
\frac{1}{8\pi}\int_{\mathbb{R}^3}\left(\left|E(t)\right|^2+\left|B(t)\right|^2\right)\phi_R - \frac{1}{8\pi}\int_{\mathbb{R}^3}\left(\left|E(0)\right|^2+\left|B(0)\right|^2\right)\phi_R\nonumber\\ 
=\frac{1}{4\pi}\int^t_0\int_{\mathbb{R}^3}\left(E\times B\right)\cdot\nabla\phi_R-\int^t_0\int_{\mathbb{R}^3}j\cdot E\phi_R.
\end{eqnarray}

Finally, since for some constant $C_T$ that does not depend on $R$ \begin{displaymath}
\left|\int^t_0\int_{\mathbb{R}^3}\left(E\times B\right)\cdot\nabla\phi_R\right|\leq \frac{C_T}{R}\left\|E\right\|_{L^{\infty,2}_{t,x}}\left\|B\right\|_{L^{\infty,2}_{t,x}},
\end{displaymath}  
it is easy to check that (\ref{Poynting}) follows from (\ref{Sequence Poynting}) by letting $R\rightarrow\infty$. The proof of the lemma is complete.
\end{proof}

\section{Energy Balance for the Vlasov Equation}
\label{Section Vlasov}

In this section we deduce the duality formula \cite{Duality} resulting from the Vlasov equation (\ref{Vlasov Equation}) and (the identity) $K\cdot\nabla_p\sqrt{1+\left|p\right|^2}\equiv v\cdot E$, which gives the energy balance associated to the Vlasov equation. Since we now face a nonlinear term in (\ref{Vlasov Equation}), we need to first prove the following lemma, a particular case of Lemma 3.5 in \cite{Bouchut}.

 \begin{lemma}
 \label{Lemma Bouchut}
  Let $\kappa_{\epsilon_1}$ and $\kappa_{\epsilon_2}$ be two regularization kernels defined on $\mathbb{R}^3_x$ and $\mathbb{R}^3_p$ respectively. Let $(f,\;E,\;B)$ be a weak solution of the RVM system. If $E,B\in \left[L^1(]0,T[;BV_{loc}(\mathbb{R}^3))\right]^3$, then there exist two sequences $\epsilon^n_1>0$, $\epsilon^n_2>0$, $\epsilon^n_1\rightarrow0$, $\epsilon^n_2\rightarrow0$ such that  
   \begin{eqnarray}
    \label{Commutator}
    \nabla_x\cdot\left[v(\kappa_{\epsilon^n_1}\kappa_{\epsilon^n_2}\ast f)\right]+\nabla_p\cdot\left[K(\kappa_{\epsilon^n_1}\kappa_{\epsilon^n_2}\ast f)\right]\nonumber\hspace{3.cm}\\
-(\nabla_x\cdot\left[vf\right])\ast\kappa_{\epsilon^n_1}\kappa_{\epsilon^n_2}-(\nabla_p\cdot\left[Kf\right])\ast\kappa_{\epsilon^n_1}\kappa_{\epsilon^n_2}
  \end{eqnarray}
converges to $0$ in $L^1(]0,T[;L^1_{loc}(\mathbb{R}^3\times\mathbb{R}^3))$.
 \end{lemma}

 \begin{proof}
First we omit the dependence in time and show the corresponding convergence on $\mathbb{R}^3\times\mathbb{R}^3$. Then we study the convergence on time as well.  
  
Indeed, the compact support of the mollifiers and the divergence theorem allow us to rewrite (\ref{Commutator}) as
\begin{eqnarray}
I^v(x,p) + I^K(x,p)\hspace{9.5cm}\nonumber\\
:=\int\int[f(x,p)-f(x-y,p-q)][v(p)-v(p-q)]\cdot\nabla_y\kappa_{\epsilon_1}(y)\kappa_{\epsilon_2}(q)dydq\hspace{1.5cm}\nonumber\\ 
+\int\int[f(x,p)-f(x-y,p-q)][K(x,p)-K(x-y,p-q)]\cdot\nabla_q\kappa_{\epsilon_2}(q)\kappa_{\epsilon_1}(y)dydq.\hspace{-.1cm}\nonumber
\end{eqnarray}
In addition, since we have that 
\begin{eqnarray}
\label{Lorentz Difference}
 K(x,p)-K(x-y,p-q) & = & E(x)-E(x-y)+v(p)\times[B(x)-B(x-y)]\nonumber\\
 & & +[v(p)-v(p-q)]\times B(x-y),
\end{eqnarray}
we may decompose the second integral by
\begin{displaymath}
 I^K=I^{K,x}+I^{K,p}
\end{displaymath} 
where $I^{K,x}$ involves the first two terms in the right side of (\ref{Lorentz Difference}) and $I^{K,p}$ involves the third term. Now, let $R>0$ and define the set $B_R\times B_R=:\Omega\subset\mathbb{R}^3\times\mathbb{R}^3$ such that $\overline{\Omega}$ + supp $\kappa_{\epsilon_1}\kappa_{\epsilon_2}$ $\subset B_{R+1}\times B_{R+1}$. In view of the assumptions of the lemma  
\begin{displaymath}
  \begin{tabular}{cc}  $\left\|E(x)-E(x-y)\right\|_{L^1_x(B_R)}\leq\left\|\nabla_xE\right\|_{\mathcal{M}(B_{R+1})}\left|y\right|$,& $\ \ \left|y\right|<\epsilon_1$,
  \end{tabular}
\end{displaymath}
(similarly for $B$), where $\left\|\nabla_xE\right\|_{\mathcal{M}(B_{R+1})}<\infty$ denotes the norm of the measure $\nabla_xE$ (resp. $\nabla_xB$), which coincides with the variation of $E$ (resp. $B$) on the ball $B_{R+1}$. Hence, since the relativistic velocity $v\in\left[C^{\infty}_b(\mathbb{R}^3)\right]^3$ satisfies $\left|v\right|\leq1$, we find that for some positive constant $C_R$ that depends on $R$   
  \begin{eqnarray}
  \label{Integral Field X}  
\left\|I^{K,x}\right\|_{L^1_{x,p}(\Omega)}&\leq& C_R\frac{\epsilon_1}{\epsilon_2}\left(\left\|\nabla_xE\right\|_{\mathcal{M}(B_{R+1})}+\left\|\nabla_xB\right\|_{\mathcal{M}(B_{R+1})}\right)\nonumber\\
&&\times \left(\int\left|\epsilon_2\nabla_q\kappa_{\epsilon_2}\right|\right)\sup_{\left|y\right|\leq\epsilon_1,\left|q\right|\leq\epsilon_2}\left\|f(x,p)-f(x-y,p-q)\right\|_{L^{\infty}_{x,p}(\Omega)}. \end{eqnarray}
Similarly, we find the estimates
  \begin{eqnarray}
  \label{Integral Field P}  
\left\|I^{K,p}\right\|_{L^1_{x,p}(\Omega)}&\leq&\left\|B\right\|_{L^2_x}\left\|\nabla_pv\right\|_{L^2_p(B_{R+1})}\nonumber\\
&&\times \left(\int\left|\epsilon_2\nabla_q\kappa_{\epsilon_2}\right|\right)\sup_{\left|y\right|\leq\epsilon_1,\left|q\right|\leq\epsilon_2}\left\|f(x,p)-f(x-y,p-q)\right\|_{L^2_{x,p}(\Omega)}   
 \end{eqnarray}
and
  \begin{eqnarray}
  \label{Integral Velocity}
\left\|I^v\right\|_{L^1_{x,p}(\Omega)} &\leq& C_R\frac{\epsilon_2}{\epsilon_1}\left(\int\left|\epsilon_1\nabla_y\kappa_{\epsilon_1}\right|\right)\left\|\nabla_pv\right\|_{L^2_p(B_{R+1})}\nonumber\\   
&&\times\sup_{\left|y\right|\leq\epsilon_1,\left|q\right|\leq\epsilon_2}\left\|f(x,p)-f(x-y,p-q)\right\|_{L^2_{x,p}(\Omega)}.   
\end{eqnarray}
Now, we have $\left(\int\left|\epsilon\nabla\kappa_{\epsilon}\right|\right)\leq C$, and we also have that 
$$
\sup_{\left|y\right|\leq\epsilon_1,\left|q\right|\leq\epsilon_2}\left\|f(x,p)-f(x-y,p-q)\right\|_{L^2_{x,p}(\Omega)}\rightarrow0,\hspace{.5cm}\texttt{as}\hspace{.3cm}\epsilon_1,\epsilon_2\rightarrow0.
$$
Hence, we can choose two sequences $\epsilon^n_1,\epsilon^n_2\rightarrow0$ with $\epsilon^n_1/\epsilon^n_2=1/n$ such that for some $n$ sufficiently large the right-hand sides of (\ref{Integral Field P}) and (\ref{Integral Velocity}) are less than $1/n$. Therefore, since we also have $f\in L^{\infty}(\Omega)$, it follows that (\ref{Integral Field X}), (\ref{Integral Field P}) and  (\ref{Integral Velocity}) go to zero as $n\rightarrow\infty$, and so does (\ref{Commutator}) in $L^1_{loc}(\mathbb{R}^3\times\mathbb{R}^3)$.

Finally, we consider the dependence in time. The difficulty here seems to arise because the sequences $\epsilon^n_1$ and $\epsilon^n_2$ may also depend on $t$. Otherwise we could just invoke the Lebesgue dominated convergence theorem as (\ref{Integral Field X}) and (\ref{Integral Velocity}) suggest. In particular, we must be careful with the estimate (\ref{Integral Velocity}). Nevertheless, if we keep track of the time dependence along the calculation, we find that
\begin{displaymath}
\left\|I^v\right\|_{L^{1,1}_{t,x,p}((0,T)\times\Omega)}\leq C_R\frac{\epsilon_2}{\epsilon_1}\sup_{\left|y\right|\leq\epsilon_1,\left|q\right|\leq\epsilon_2}\left\|f(t,x,p)-f(t,x-y,p-q)\right\|_{L^{1,2}_{t,x,p}((0,T)\times\Omega)}   
\end{displaymath}
and we can reason as above. This concludes the proof of the lemma.
\end{proof}

We now turn to the energy balance (\ref{Formal Energy Vlasov}) associated to the Vlasov equation. 

\begin{lemma}
\label{Energy Balance Vlasov}
Let $\lambda>0$. In addition to the assumptions of Lemma \ref{Lemma Bouchut}, suppose that $\rho_{\lambda}$ as defined in (\ref{Rho Lambda}) is in $L^{\infty}(]0,T[;L^2(\mathbb{R}^3))$. Then
\begin{equation}
\label{Energy Balance}
\int_{\mathbb{R}^3\times\mathbb{R}^3}\sqrt{1+\left|p\right|^2}f(t)dxdp=\int_{\mathbb{R}^3\times\mathbb{R}^3} \sqrt{1+\left|p\right|^2}f(0)dxdp + \int^t_0\int_{\mathbb{R}^3}E\cdot jdsdx
\end{equation}
for almost all $t\in[0,T[$.
\end{lemma}

\begin{proof} $(f,\;E,\;B)$ is a weak solution of the RVM system. Thus, as a straightforward consequence of Lemma \ref{Lemma Bouchut}, there are two sequences $\epsilon^n_1>0$, $\epsilon^n_2>0$, $\epsilon^n_1\rightarrow0$, $\epsilon^n_2\rightarrow0$ such that         
 \begin{equation}   
 \label{Modified Vlasov}
 \partial_tf^n+v\cdot\nabla_xf^n+K\cdot\nabla_pf^n=r^n
 \end{equation}
converges to $0$ in $L^1(]0,T[;L^1_{loc}(\mathbb{R}^3\times\mathbb{R}^3))$ as $n\rightarrow\infty$, where $f^n:=\kappa_{\epsilon^n_1}\kappa_{\epsilon^n_2}\ast f$ and $r^n$ is defined by (\ref{Commutator}). 

Consider a family of smooth cut-off functions $\phi_R=\phi(\frac{\cdot}{R})$, $R\geq1$ where $\phi\in C^{\infty}_0(\mathbb{R}^6)$, $\phi\geq0$ and $\phi\equiv1$ on $B_1$ $\subset$ supp$\phi$ $\subset B_2$. If we multiply (\ref{Modified Vlasov}) by $\sqrt{1+\left|p\right|^2}\phi_R$ and integrate by parts, we find that  
 \begin{eqnarray}
 \label{Sequences}
\int_{\mathbb{R}^6}\sqrt{1+\left|p\right|^2}f^n(t)\phi_R - \int_{\mathbb{R}^6} \sqrt{1+\left|p\right|^2}f^n(0)\phi_R\hspace{3.cm}\nonumber\\ 
=  \int^t_0\int_{\mathbb{R}^6}E\cdot vf^n\phi_R+\int^t_0\int_{\mathbb{R}^6}\sqrt{1+\left|p\right|^2}\phi_R r^n\hspace{2.0cm}\nonumber\\ 
+ \int^t_0\int_{\mathbb{R}^6}\sqrt{1+\left|p\right|^2}f^nK\cdot\nabla_p\phi_R+ \int^t_0\int_{\mathbb{R}^6}\sqrt{1+\left|p\right|^2}f^nv\cdot\nabla_x\phi_R.\hspace{-.6cm}
\end{eqnarray} 
Here we have used the identity $K\cdot\nabla_p\sqrt{1+\left|p\right|^2}\equiv v\cdot E$.

Let $n\rightarrow\infty$. In doing so, we notice that the second term in the right-hand side vanishes as a consequence of Lemma \ref{Lemma Bouchut}. Also, the convergence of the two terms in the left-hand side and the last term in the right-hand side follow by a straightforward application of the theorem of smooth approximations. Thus, we are led to prove the convergence of the first and third terms in the right-hand side.

Indeed, the reasoning done so far does not preclude us from writing $\phi_R$ as the product of two suitable functions $\chi_R=\chi(\frac{\cdot}{R})$ and $\zeta_R=\zeta(\frac{\cdot}{R})$ where $\chi\in C^{\infty}_0(\mathbb{R}_x^3)$ and $\zeta\in C^{\infty}_0(\mathbb{R}_p^3)$. Hence, 
\begin{eqnarray}
\left|\int_{\mathbb{R}^6}E\cdot v(f-f^n)\phi_R\right|&\leq&\int_{\mathbb{R}^3}\left|E\right|\chi_R\int_{\mathbb{R}^3}\zeta_R\left|f-f^n\right|\nonumber\\
& \leq & C_R\left\|E\right\|_{L^2_x}\left\|f-f^n\right\|_{L^2_{x,p}},\nonumber
\end{eqnarray}     
which converges to zero as $n\rightarrow\infty$. Then, a use of the Lebesgue dominated convergence theorem provide the convergence of the first term in the right-hand side. Since we can do similarly with the remaining term, we find that as $n\rightarrow\infty$, (\ref{Sequences}) converges to  
\begin{eqnarray}
\int_{\mathbb{R}^6}\sqrt{1+\left|p\right|^2}f(t)\phi_R
=\int_{\mathbb{R}^6} \sqrt{1+\left|p\right|^2}f(0)\phi_R+\int^t_0\int_{\mathbb{R}^6}E\cdot vf\phi_R\hspace{.7cm}\nonumber\\ 
+ \int^t_0\int_{\mathbb{R}^6}\sqrt{1+\left|p\right|^2}fK\cdot\nabla_p\phi_R+ \int^t_0\int_{\mathbb{R}^6}\sqrt{1+\left|p\right|^2}fv\cdot\nabla_x\phi_R.\hspace{-.7cm}\nonumber
\end{eqnarray}

Finally, we let $R\rightarrow\infty$ and show that the above equality converges to (\ref{Energy Balance}). The convergences of the term in the left and the first term in the right-hand side are straightforward, since for $t=0$ and for almost all $t>0$, $f(t)\in L^1_{kin}(\mathbb{R}^6)$. Also, since
$$ \left|\int^t_0\int_{\mathbb{R}^6}\sqrt{1+\left|p\right|^2}fv\cdot\nabla_x\phi_R\right|\leq\frac{C}{R}\int^t_0\int_{\mathbb{R}^6}\sqrt{1+\left|p\right|^2}f\leq\frac{C_T}{R},
$$
the last term converges to zero as $R\rightarrow\infty$. In order to obtain the convergence of the second term in the right, we first notice that for any $\lambda>0$, 
\begin{eqnarray}
  \rho(t,x) & = & \int_{\left|p\right|\leq1}f(t,x,p)dp+\int_{\left|p\right|>1}f(t,x,p)dp\nonumber\\
  & \leq & \sqrt{4\pi/3}\left\|f(t,x)\right\|_{L^2_p}+\rho_{\lambda}(t,x)\nonumber.
\end{eqnarray}
Then, the hypothesis made in the lemma implies that 
\begin{equation}
\label{Estimate}
 \left\|\rho(t)\right\|_{L^2_x}\leq \sqrt{4\pi/3}\left\|f(t)\right\|_{L^2_{x,p}}+\left\|\rho_{\lambda}(t)\right\|_{L^2_x}<\infty.
\end{equation}
As a result, and since $\left|v\right|\leq 1$, we can easily verify that $E\cdot vf\in L^1(]0,T[\times\mathbb{R}^6)$, so the Lebesgue theorem provides the expected convergence. Hence, we are only left to show that the third term in the right-hand side converges to zero. To this end, we first produce 
\begin{equation}
\label{Here And Now}
\left|\int^t_0\int_{\mathbb{R}^6}\sqrt{1+\left|p\right|^2}fK\cdot\nabla_p\phi_R\right|\leq\frac{C}{R}\int^t_0\int_{\mathbb{R}^3}\left(\left|E\right|+\left|B\right|\right)\int_{R\leq\left|p\right|\leq 2R}\left|p\right|f.
\end{equation}
To estimate the above inequality we observe that 
$$ \frac{1}{R}\int_{R\leq\left|p\right|\leq 2R}\left|p\right|f\leq\left\{\begin{array}{lll}
       2^{1-\lambda}\rho_{\lambda}/R^{\lambda} &, & 0<\lambda<1 \\   
       \rho_{\lambda}/R^{\lambda} &, & 1\leq\lambda \end{array}.\right. 
$$
Thus, for any $\lambda>0$, there exists a constant $C>0$ independent of $R$ such that the right-hand side of (\ref{Here And Now}) is less or equal than 
$$ \frac{C}{R^{\lambda}}\int^t_0\left(\left\|E(s)\right\|_{L^2_x}+\left\|B(s)\right\|_{L^2_x}\right)\left\|\rho_{\lambda}(s)\right\|_{L^2_x}ds\leq\frac{C_T}{R^{\lambda}}.
$$
Therefore, (\ref{Here And Now}) converges to zero as $R\rightarrow\infty$ and the proof of the lemma is complete.
\end{proof}

\section{Proof of Theorem \ref{The Theorem}}

\begin{proof}
Since $\left|v\right|\leq1$, we have $\left|j\right|\leq\rho$. Then, in view of (\ref{Estimate}), we can combine Lemmas \ref{Lemma 1} and \ref{Energy Balance Vlasov} to produce the equality for almost all $0\leq t<T$ claimed for (\ref{Bound Total Energy}).
\end{proof}


          %







          %

\medskip

{\bf Acknowledgement.} I am grateful to Prof. R. Illner for useful discussions and insightful comments concerning the subject matter of this paper.

\end{document}